\documentclass[11pt]{article}
\usepackage{amsfonts}
\usepackage{amsmath}
\usepackage{latexsym}

\usepackage{amsthm}
\usepackage[bookmarks]{hyperref}
\usepackage{paralist}
\usepackage{color}
\usepackage{soul}
\definecolor{lightblue}{rgb}{.60,.60,1}

\newcommand{\handout}[5]{
   \noindent
   \begin{center}
   \framebox{
      \vbox{
    \hbox to 5.78in { {\bf Draft- do not distribute} \hfill #2 }
       \vspace{4mm}
       \hbox to 5.78in { {\Large \hfill #5  \hfill} }
       \vspace{2mm}
       \hbox to 5.78in { {\it #3 \hfill #4} }
      }
   }
   \end{center}
   \vspace*{4mm}
}

\newtheorem{theorem}{Theorem}

\newtheorem{lemma}[theorem]{Lemma}
\newtheorem{question}[theorem]{Question}
\newtheorem{proposition}[theorem]{Proposition}
\newtheorem{definition}[theorem]{Definition}
\newtheorem*{championlemma}{Champions Lemma}
\theoremstyle{remark}

\newtheorem*{remark}{Remark}

\newcommand{\size}[1]{{\left|#1\right|}}
\newcommand{\pair}[1]{{\left\langle#1\right\rangle}}
\renewcommand{\H}{\mathsf{H}}
\renewcommand{\L}{\mathsf{L}}
\newcommand{\MIN}{\mathrm{MIN}}

\newcommand{\HIGH}{\mathrm{HIGH}}
\renewcommand{\phi}{\varphi}
\DeclareMathOperator{\poly}{poly}
\newcommand{\0}{\emptyset}
\newcommand{\T}{\mathrm{T}}
\newcommand{\N}{\mathbb{N}}
\newcommand{\join}{\mathrel{\oplus}}
\newcommand{\converge}{\mathop\downarrow}
\newcommand{\diverge}{\mathop\uparrow}
\newcommand{\complement}{\overline}
\renewcommand{\min}{{\textstyle\mathop{\mathrm{min}}}}

\numberwithin{equation}{section}

\title{On approximate decidability of minimal programs}

\author{
Jason Teutsch\footnote{Research supported by the Centre for Quantum Technologies, Institute for Mathematical Sciences of the National University of Singapore, NUS grant number R146-000-181-112, and a Fulbright Fellowship provided by the United States-Israel Educational Foundation.} \\
\emph{Ben-Gurion University} \\
\url{teutsch@cs.uchicago.edu}
\and 
Marius Zimand\footnote{The author was supported in part
by NSF grant CCF 1016158.}\\
\emph{Towson University}\\
\url{mzimand@towson.edu}
}

\begin{document}

\maketitle

\begin{abstract}
An index $e$ in a numbering of partial-recursive functions is called minimal if every lesser index computes a different function from $e$.  Since the 1960's it has been known that, in any reasonable programming language, no effective procedure determines whether or not a given index is minimal.  We investigate whether the task of determining minimal indices can be solved in an approximate sense.  Our first question, regarding the set of minimal indices, is whether there exists an algorithm which can correctly label 1 out of $k$ indices as either minimal or non-minimal.  Our second question, regarding the function which computes minimal indices, is whether one can compute a short list of candidate indices which includes a minimal index for a given program.  We give some negative results and leave the possibility of positive results as open questions.
\end{abstract}

\section{Occam's razor for algorithms}

In any reasonable programming system, one can code a computable function many different ways.  The shortest such code has  practical value, theoretical significance, and philosophical allure.  Despite their implicit appeal to simplicity, shortest codes remain elusive.  Indeed no algorithm can enumerate more than finitely many shortest codes \cite{Blu67,Sch98}, and, as a particular consequence, there is no effective way to obtain or to recognize a shortest code.\footnote{For 
background on the set of minimal codes, see the survey article \cite{Sch98} and the more recent articles \cite{JST11, ST08, Teu07}.}

Shortest descriptions of finite strings exhibit similar ineffectivity phenomena. It is well known that any algorithm can enumerate at most finitely many strings of high complexity \cite{LV08}, that no unbounded computable function is a lower bound for Kolmogorov complexity \cite{ZL70}, and that any algorithm mapping a string to a list of values containing its Kolmogorov complexity must, for all but finitely many lengths~$n$, include in the list for some string of length~$n$ at least a fixed fraction of the lengths below $n+O(1)$ \cite{BBFFGLMST06}.  This paper adds to this list of inapproximability results that the set of Kolmogorov random strings is not $(1,k)$-recursive for any~$k$ (Theorem~\ref{thm: howfatisshe}).

In contrast,  several recent works~\cite{BMVZ13,BZ14,Teu14, Zim14} have revealed that a certain type of approximation, called \emph{list approximation}, can be obtained for shortest descriptions of strings in a surprisingly efficient way. In list approximation, instead of achieving the ideal objective of constructing an object that has some coveted property, we are content if we at least can construct a short list (of ``suspects") guaranteed to contain such an object. Bauwens, Makhlin, Vereshchagin, and Zimand~\cite{BMVZ13} showed that there is a computable function which maps strings to quadratic-length lists of strings such that one element in the list is a description of the given string with minimal length (within an additive constant).  Moreover, similar lists with slightly weaker parameters can actually be constructed in polynomial time. Teutsch~\cite{Teu14}, Zimand~\cite{Zim14} and Bauwens and Zimand~\cite{BZ14} have obtained polynomial-time constructions with improved parameters. 

In the wake of these recent positive results for finite strings, it is natural to ask whether such list approximations transfer to the semantic case of partial-recursive functions. This paper investigates this question and also the $(1,k)$-recursiveness of the set of minimal programs, which might be viewed as the decidability analog of list approximation for sets.

The shortest program of a function depends on the programming system. Formally, a programming system is a \emph{numbering} $\phi$, given by a partial-recursive function $U$ mapping pairs of natural numbers to natural numbers. For every $e \geq 0$, the $e$-th function in the numbering, denoted $\phi_e$, is defined as $\phi_e(x) = U(e,x)$. The main objects of interest in this paper are defined as follows.
\begin{definition} Let $\phi$ be a numbering.
\begin{enumerate}[\scshape (i)]
\item
$\MIN_\phi = \{e \colon (\forall j<e)\: [\phi_j \neq \phi_e]\}$
is the \emph{set of $\phi$-minimal indices}.  
\item
The function $\min_\phi(e)$ denotes the unique index $j \in \MIN_\phi$ such that $\phi_j = \phi_e$.
\end{enumerate}
\end{definition}
\paragraph{Numberings.}  There are many effective ways to enumerate partial-recursive functions.  If we want a meaningful notion of ``shortest program,'' then at the very least we should consider only \emph{universal} numberings, that is, numberings which include all partial-recursive functions.  This class of numberings still contains pathologies, such as the Friedberg numberings  \cite{Fri58, Kum90}.  In a Friedberg numbering every partial-recursive function has a unique index, and hence every index is minimal.  In computability theory one typically uses \emph{acceptable numberings} (also known as \emph{G\"{o}del numberings}), and previous studies of minimal indices focused on this type of numberings.  We recall that a numbering $\phi$ is called \emph{acceptable} if for every further numbering $\psi$ there exists a recursive function $f$ (the ``translation function'') such that $\phi_{f(e)} = \psi_e$ for all $e$. 

During this investigation we have observed that, despite appeals to the Recursion Theorem and hardness of index sets in prior literature, the effectiveness of the translation function for acceptable numberings seems to serve as a red herring when dealing with minimal indices.  Many basic results involving acceptable numberings continue to hold if we merely require a computable bound on the output of the translation function. Consequently, we introduce the following type of numbering.
\begin{definition}
\label{d:cbnumbering}
A numbering $\phi$ is called \emph{computably bounded} if for any further numbering $\psi$, there exists a computable function $f$ such that for any $e$, $\phi_j = \psi_e$ for some $j \leq f(e)$.
\end{definition}

Robust results require an absolute definition of ``shortest program'' in which computability properties do not depend on the underlying numbering.  \emph{Kolmogorov} numberings capture this notion.  An acceptable numbering $\phi$ is called a \emph{Kolmogorov} numbering if for every further numbering $\psi$ the corresponding \emph{translation} function $f$ is linearly bounded, that is there exists a positive constant $c$ such that for every index $e$, $f(e) \leq ce + c$.  The standard universal Turing machine \cite{Soa87,Tur36} is an example of a Kolmogorov numbering.  The point is that if we define $\size{e}$, the length of the ``program" $e$,  as $\log e$, then $\size{f(e)} \leq \size{e} + O(1)$, and therefore if $\phi$ is a Kolmogorov numbering and $\psi$ is an arbitrary numbering, then for every index $e$, $\size{\min_\phi(e)} \leq \size{\min_\psi(e)} +O(1)$ where the constant $O(1)$ depends only on the linear translation function from $\psi$ to $\phi$.   These same inequalities hold when the translation function~$f$ is linear but not necessarily recursive. Such a numbering, where the translation function is linearly bounded but not necessarily recursive, is said to have the \emph{Kolmogorov property}.  Numberings with the Kolmogorov property are sometimes referred to as \emph{optimal} because, as we have observed with universal machines for Kolmogorov complexity, any numbering can be translated into it with a constant overhead increase in program length.  Previously, Jain, Stephan, and Teutsch \cite{JST11} investigated Turing degrees for sets of minimal indices with respect to numberings with the Kolmogorov property.

We conclude our discussion on numberings by separating the notions of acceptable numbering and numbering with the Kolmogorov property.  The reader may want to read first the ``Notation and basic prerequisites'' paragraph at the end of this section.

\begin{proposition}[Stephan]
A numbering with the Kolmogorov property need not be acceptable, and vice versa.
\end{proposition}
\begin{proof}
First, let us observe that an acceptable numbering need not have the Kolmogorov property.  For any acceptable $\phi$, the numbering $\psi$ given by $\psi_{2^e} = \phi_e$ and where non-powers of two $\psi$-indices give the everywhere divergent function is an acceptable numbering which does not have the Kolmogorov property.  Indeed, for any index $e \in \MIN_\phi$ besides the minimal index for the everywhere divergent function, the least $\psi$-index which computes $\phi_e$ is $2^e$.  Therefore no linearly bounded translation function from $\phi$ to $\psi$ exists.

Next we construct a numbering with the Kolmogorov property which is not acceptable.  Let $A = \lim A_s$ be a limit-computable \cite{Soa87}, bi-immune set, for example Chaitin's Omega \cite{DH10,LV08}.  Let $\phi$ be a numbering with the Kolmogorov property, and define a further numbering~$\psi$ by
\[
\psi_{2e}=
\begin{cases}
\phi_e &\text{if $e \in A$, and} \\
\text{some function with finite domain} &\text{otherwise,}
\end{cases}
\]
and similarly for $\psi_{2e+1}$ but with ``$e \in A$'' replaced with ``$e \notin A$.''  This can be done by setting $\psi_{2e,s} = \phi_{e,s}$ at stages~$s$ where $e \in A_s$ and making $\psi_{2e+1,s} = \phi_{e,s}$ when $e \notin A_s$.  In the other stages, we simply freeze the computations of $\psi_{2e}$ or $\psi_{2e+1}$.  Since $A$ is limit-computable, after some finitely stage, one of $\psi_{2e}$ and $\psi_{2e+1}$ becomes permanently frozen and the other function goes on to compute $\phi_e$.

Now the translation bound for $\psi$ is at most twice the translation bound for $\phi$, hence $\psi$ has the Kolmogorov property.  On the other hand, $\psi$ cannot be acceptable.  Suppose that there were some computable function~$f$ such that $\psi_{f(x)}$ is the everywhere constant~$x$ function.  Note that if $f(x)$ is even, $f(x)/2 \in A$, and if $f(x)$ is odd, then $(f(x)-1)/2 \notin A$.  Now infinitely often the index $f(x)$ must either be even or odd, contradicting that $A$ is bi-immune.
\end{proof}

\paragraph{Our results and paper roadmap.} As already mentioned, we seek to understand to what extent and for which type of numberings $\min_\phi$ is list approximable and/or $\MIN_\phi$ is $(1,k)$-recursive. The notion of list approximability of functions has been already explained.  We now define the other type of approximability in a slightly generalized form.  Below we use $\chi_A$ to denote the characteristic function for the set~$A$.

\begin{definition} A set of integers $A$ is called $(m,k)$-recursive if there exists a computable function mapping $k$-tuples of strings to labels $\{0,1\}^k$ such that for every tuple $(x_1, \ldots, x_k)$  the vectors $(\chi_A(x_1), \dotsc, \chi_A(x_k))$ and $f(x_1, \dotsc, x_k)$ coincide in at least $m$ positions (i.e., at least $m$ of the labels are correct if we interpret $0$ as ``not-in $A$" and $1$ as ``in A.")
\end{definition}

We formalize the two main problems investigated in this paper, and describe our contribution.  In his Masters thesis, a survey article on minimal indices, Schaefer posed the following problem.
\begin{question}[\cite{Sch98}] \label{ques: schaefer's problem}
Does there exist an acceptable numbering $\phi$ and a positive integer~$k$ such that $\MIN_\phi$ is $(1,k)$-recursive?
\end{question}
Schaefer showed that there exists a Kolmogorov numbering $\psi$ such that $\MIN_\psi$ is not $(1,k)$-recursive for any $k$. In Section~\ref{s:minset}, we extend his existential result to all numberings with the Kolmogorov property (Theorem~\ref{t:onekrec}).  Schaefer showed that for any acceptable numbering $\phi$, $\MIN_\phi$ is not $(1,2)$-recursive, but his original problem for the case $k > 2$ remains open.

The second problem, which we dub the \emph{shortlist problem for functions} is as follows:
\begin{question} \label{ques: poly-log list}
Let $\phi$ be a numbering with the Kolmogorov property.   Does there exist a computable function which maps each index $e$ to a $\poly(\size{e})$-size list containing an index $j$ such that $\phi_e = \phi_j$ and $\size{j} \leq \size{\min_\phi(e)} + O(\log \size{e})$?
\end{question}
Ideally, we would like to replace the overhead ``$O(\log \size{e})$'' above with ``$O(1)$,'' however determining whether either of these bounds is possible appears to be outside the reach of present techniques. By ``polynomial-size list'' we mean a finite set having cardinality bounded by a polynomial in $\size{e}$, but the question is interesting for any non-trivial list size.

Our main results are proved in Section~\ref{s:minf}:
\begin{itemize}
\item If $\phi$ is an acceptable numbering and a computable function on input $e$ returns a list containing the minimal $\phi$-index for $\phi_e$, then the size of that list cannot be constant (Theorem~\ref{thm:listlength2}).

\item For every numbering $\phi$  with the Kolmogorov property, if a computable function on input $e$ returns a list containing the minimal $\phi$-index for $\phi_e$, then  the size of the list must be  $\Omega(\log^2 e)$ (Theorem~\ref{t:quadbound}).

\item There exists a Kolmogorov numbering~$\phi$ such that if  a computable function on input $e$ returns a list containing the minimal $\phi$-index for $\phi_e$, then  the size of that list must be  $\Omega(e)$ (Theorem~\ref{t:expbound}).
\end{itemize}
In summary, our results show that a computable list that contains the minimal index cannot be too small.  Along the lines of the second result  (Theorem~\ref{t:quadbound}), we formulate the following question whose finite string version has a positive answer~\cite{BMVZ13}: 
\begin{question} \label{ques: quadshort}
Does there exist a Kolmogorov numbering $\phi$ with a computable list that contains $\min_\phi(e)$ and has size $O(\log^2 e)$?
\end{question}
A positive answer to Question~\ref{ques: quadshort} would immediately yield a positive answer to Question~\ref{ques: poly-log list} when restricted to Kolmogorov numberings but not necessarily in general since the Kolmgorov property does not permit us to translate indices effectively.

These results have analogues if, roughly speaking, we substitute (computable function, minimal index) with
(finite string, Kolmogorov complexity) and they are easier to establish in the latter setting.  We obtain the results above by building bridges between the two settings, as we explain in Section~\ref{s:tech}, with several technical lemmas. In particular,  Lemma~\ref{lem:pdl} provides a connection between between Schaefer's problem and the shortlist problem for functions  (Question~\ref{ques: schaefer's problem} and Question~\ref{ques: poly-log list}), as this lemma is used in the proofs for both Theorem~\ref{t:onekrec} and Theorem~\ref{t:expbound}.  

Finally, in Section~\ref{s:compbound}, we extend some results from the literature regarding $\MIN_\phi$ from acceptable numberings to computably bounded numberings. If $\phi$ is an acceptable numbering, it is known that $\MIN_\phi$ is immune~\cite{Blu67}, is Turing equivalent to the jump of the halting problem~\cite{Mey72}, and is not $(1,2)$-recursive~\cite{Sch98}. We show that all this properties continue to hold if $\phi$ is a computably bounded numbering.

\paragraph{Notation and basic prerequisites.} For sets $A$ and $B$, we write $A \leq_\T B$ if $A$ is Turing reducible to $B$, that is, if $A$ can be computed using the set~$B$ as an oracle, and we say $A$ is $B$-computable.  $A \join B$ denotes the set $\{2x \colon x \in A\} \cup \{2x+1 \colon x \in B\}$. $K$ is the halting set.  A set $A$ is called \emph{$\Sigma^0_2$} if there exists computable predicate~$P$ such that $x \in A \iff (\exists y)\: (\forall z)\: [P(x,y,z)]$, and $A$ is \emph{$\Pi^0_2$} if a similar condition holds with the quantifiers reversed.  $A$ is $\Delta^0_2$ if it is both $\Sigma^0_2$ and $\Pi^0_2$, or equivalently by Post's Theorem, $A \leq_\T K$. 
An infinite set is called \emph{immune} if there is no algorithm which enumerates infinitely many of its members.  A set is \emph{bi-immune} if both the set and its complement are immune.  A set $A$ is limit-computable if there exists a computable sequence of sets $A_s$ such that $\lim_s A_s = A(x)$.  Throughout this paper, we use $\phi_{e,s}$ to denote the computation of $\phi_e$ up to $s$~steps, and it may happen that $\phi_e$ does not converge within $s$~steps.
 For more details on these notions, see \cite{Soa87}.  

Throughout this exposition, we fix a universal machine~$U$ and use $C(x) = \min \{\size{p} \colon U(p) = x\}$ to denote the Kolmogorov complexity of a string~$x$.  Here $\size{p}$ denotes $\log p$, the length of the program~$p$.  Similarly we will use $C(x \mid y) = \min \{\size{p} \colon U(p,y) = x\}$ to denote the \emph{conditional complexity of~$x$ given~$y$}.   Further background on Kolmogorov complexity can be found in the standard textbook by Li and Vit\'{a}nyi \cite{LV08}, as well as the forthcoming textbook by Shen and Vereshchagin~\cite{SV14}.

 An \emph{order} is an unbounded, nondecreasing function from $\N$ to $\N$.

\section{Proof techniques linking the shortest descriptions of strings and functions}\label{s:tech}
We start with a result that illustrates some of the proof techniques that we use later in more complicated settings.

\paragraph{Warm up: a winner-goes-on tournament.} \label{sec:wgo}
We show here that the \emph{set of strings with randomness deficiency~1}, $D = \{ x \colon C(x) \geq \size{x}-1\}$, is not $(1,2)$-recursive.  First, note that more than half of the strings at each length must belong to $D$ because there are $2^n$ strings of length~$n$ and at most $2^0+ 2^1 + \dotsc 2^{n-2} = 2^{n-1} - 1$ strings with randomness deficiency less than~1.  Suppose there were a computable function $f \colon \{0,1\}^2 \to \{\L,\H\}^2$ witnessing that $D$ is $(1,2)$-recursive, where ``$\L$'' is the label for low complexity (deficiency greater than 1), and ``$\H$'' is the label for high complexity.  We shall show that for every $n$, there is a string $x \in D$ of length $n$ such that $C(x) < \log n + O(1)$.

Consider the restriction of $f$ to strings of length~$n$, and suppose that some pair of binary strings $(x,y)$ of length~$n$ receives the label $(\H,\H)$ from~$f$.  Then by definition of~$f$, either $x$ or $y$ must have high complexity.  But when $(x,y)$ is the lexicographically least pair of strings of length $n$ satisfying $f(x,y) = (\H,\H)$, we can compute either $x$ or $y$ with a single advice bit given the length~$n$.  It follows that the Kolmogorov complexities for $x$ and $y$ are bounded by $\log n + O(1)$, which for sufficiently large~$n$ contradicts that one of them has randomness deficiency less than~1.

Thus it suffices to assume that the only labels which occur among pairs of binary strings with length~$n$ are $(\L,\H)$, $(\H,\L)$, and $(\L,\L)$.  We say that a set of strings $S$ of length~$n$ form a \emph{clique} if every pair of distinct vertices $(x,y) \in S$ receive either the label $(\L,\H)$ or $(\H,\L)$.  Fix $S$ as the lexicographically least clique which contains more than half of the strings of length~$n$.  Such a clique must exist because the set $D$ restricted to strings of length~$n$ forms a clique.  Furthermore any clique which contains more than half of the strings of length $n$ must contain a string of high complexity since most strings have high complexity.

At this point, we can limit our search for a complex string to the clique $S$ where only the labels $(\L,\H)$ and $(\H,\L)$ occur.  We run a ``winner-goes-on'' tournament on $S$.  A \emph{match} consists of a pair of strings, and the \emph{winner} of a match is the string labeled ``$\H$.''  The \emph{tournament} proceeds as follows.  We start with an arbitrary pair of of strings, and the winner of that match faces the next \emph{challenger}, a string which has not yet appeared in the tournament.  The winner of this second match goes on to face the subsequent challenger, and the tournament ends once all strings have appeared in the tournament at least once, that is, when we run out of new challengers.  The final winner of this tournament has high complexity.  Indeed, at some point in the tournament, a string with high complexity must enter then tournament, and thereafter the winner always has high complexity.  But we can describe this string of high complexity using $\log n + O(1)$ bits, a contradiction.  Therefore no such function~$f$ exists.
\medskip

Some of the main ideas used in the proof of the ``warm-up'' can be replicated for minimal indices. There are two key points in that proof. The first one  is that all the elements of the set we show is not $(1,2)$-recursive have high Kolmogorov complexity. In Section~\ref{s:kolmindices}, we observe that minimal indices also have relatively high Kolmogorov complexity.  The second key point is that the tournament is played among the elements of a simple set (namely the set of $n$-bit strings) that  has a high density (namely at least $1/2$) of strings with high complexity.  In Section~\ref{s:highdensity}, we show that  minimal indices in numberings that have the Kolmogorov property satisfy a similar density condition.  Finally, in Section~\ref{s:champion}, we show that an extension of the ``tournament'' argument from the ``warm-up'' can be used to show that certain sets are not $(1,k)$-recursive.

\subsection{The Kolmogorov complexity of minimal indices}
\label{s:kolmindices}
We observe that the elements of $\MIN_\phi$ have relatively high Kolmogorov complexity. 
\begin{lemma}\label{lem:okra}
 For every computably bounded numbering $\phi$, there exists a computable order~$g$ such that $C(x) \geq g(x)$ for all $x \in \MIN_\phi$.
\end{lemma}

\begin{proof}
Let $U$ be the underlying universal machine for the Kolmogorov complexity function $C$.  Define a numbering $\psi$ by
\[
\psi_q = \phi_{U(q)},
\]
let $t$ be a ``translator" program such $\phi_{t(q)} = \psi_{q}$, and let $s$ a computable function such that $t(q) \leq s(q)$ for all $q$.  Let $g$ be the computable function defined as follows: $g(x)$ is the smallest $y$ such that $s(y) \geq x$. 

Now let $x \in \MIN_\phi$ and suppose that $C(x) < g(x)$. This means that there exists $p < g(x)$ such that $U(p)=x$.
Note that $t(p) \leq s(p) < x$ and  $\phi_{t(p)} = \psi_p = \phi_{U(p)} = \phi_x$.   This contradicts that $x$ is a minimal $\phi$-index. 
\end{proof}

In the case of numberings with the Kolmogorov property we obtain better bounds.
\begin{lemma}\label{l:highmin}
If $\phi$ has the Kolmogorov property, then for all $x \in \MIN_\phi$ it holds that $C(x) \geq \size{x} - \log \size{x}$.
\end{lemma}

\begin{proof}
Let $m = \size{x}$. Since $x$ is written on $m$ bits, we have $x \geq 2^{m-1}$. Suppose $C(x) < m - \log m$. Then there exists $p$ of length less than $m - \log m$ such that $U(p)=x$. Define a numbering $\psi$ by
\[
\psi_q = \phi_{U(q)}.
\]
Since $\phi$ has the Kolmogorov property, there exists a constant $c$ such that for some
\[
v < cp + c \leq c 2^{m - \log m} + c < 2^{m-1} \leq x
\]
it holds that $\phi_v = \psi_p = \phi_{U(p)} = \phi_x$. This contradicts that $x \in \MIN_\phi$.
\end{proof}

\subsection{Sets of low complexity with high density of minimal indices}
\label{s:highdensity}
Among the strings of a given length, a large fraction are incompressible or close to incompressible. We show here that for numberings with Kolmogorov property and for a generalization of such numberings, which we call \emph{polynomially-bounded numberings}, minimal indices have a similar property.  As we shall see, there exist infinitely many finite sets of low complexity whose vast majority of elements are minimal indices.  Such families of sets will be used several times in this work.

\begin{definition}
\label{d:polynumbering}
A numbering $\phi$ is called \emph{polynomially bounded} if for any further numbering $\psi$, there exist positive integers $k$ and $c$ such that for any $e$, $\phi_j = \psi_e$ for some $j \leq ce^k + c$.
\end{definition}

We note that a numbering which is polynomially bounded with parameter $k=1$ has the Kolmogorov property and that only the case $k=1$ is used  in the rest of the  paper.  The following ``one-dimensional'' version of $\MIN_\phi$ will be used throughout our paper.

\begin{definition}
For every numbering $\phi$, we define the following subset of $\phi$-minimal indices.
\[
M_\phi = \{e : e \in \MIN_\phi \text{ and $\phi_e$ converges only on input~0}\}.
\]
\end{definition}

The following lemma estimates the fraction of indices which are minimal in a large interval.

\begin{lemma}
\label{l:polyintervalbinary}
Let $\phi$ be a polynomially bounded numbering with degree $p$, and let $d(k) = p^{k-1} + p^{k-2} + \dotsb + 1$.  Then there exists a positive integer $a$ such that for every sufficiently large $n$ the interval $I_n =\{2^{ad(n)}+1, \dotsc, 2^{ad(n+1)}\}$ satisfies
\[
\size{M_\phi \cap I_n} > 2^{-ap^{n}}\cdot \size{I_n}.
\]
When $\phi$ has the Kolmogorov property, that is when $p=1$, $M_\phi \cap I_n$ occupies at least constant fraction of the indices in $I_n$.
\end{lemma}
\begin{proof}
Let $\psi_0, \psi_1, \ldots$ be the following numbering of all p.c.\ functions which are only defined at~0: $\psi_x(0) = x$, and $\psi_x(y)\uparrow$ for all $y \geq 1$.  Let $c$ be the positive constant guaranteed by the degree~$p$ polynomial bound so that for any index~$e \geq 1$, there exists $j<ce^p$ such that $\phi_j = \psi_e$ and also $\psi_0 = \phi_j$ for some $j \leq c$.  Fix $a=p + \lceil \log c \rceil$, and let $I_n = \{2^{ad(n)}+1, \dotsc, 2^{ad(n+1)}\}$.  Each of the functions $\psi_0, \psi_1, \ldots, \psi_{2^{ad(n)+1}}$ has a minimal $\phi$-index bounded by 
\[
c (2^{ad(n)+1})^p = c 2^p 2^{p\cdot a d(n)} \leq 2^a 2^{a\cdot p d(n)} = 2^{ad(n+1)} = \max I_n.
\]
It follows that at least $(2^{ad(n)+1}+1) - \min I_n = 2^{ad(n)}$ of these minimal $\phi$-indices must lie in $I_n$.  Since
\[
\frac{\size{M_\phi \cap I_n}}{\size{I_n}} \geq \frac{2^{ad(n)}}{2^{ad(n+1)} - 2^{ad(n)}} = \frac{1}{2^{ap^{n}}-1},
\]
the conclusion follows.
\end{proof}

Next we show that with a small amount of advice (constant advice in the case of the Kolmogorov property), we can shrink the intervals from the previous lemma so that the resulting set contains a high concentration of minimal indices.

\begin{lemma}[polynomial density-boosting] \label{lem:pdl}
Let $\phi$ be a polynomially bounded numbering with degree~$p$, and let
\[
I_n =\{2^{ad(n)}+1, \dotsc, 2^{ad(n+1)}\}
\]
as in Lemma~\ref{l:polyintervalbinary}.  Then for every $\epsilon >0$ and $n \geq 0$, there exists a subset $A_n \subseteq I_n$ such that
\begin{enumerate}[\scshape (i)]
\item $\size{M_\phi \cap A_n} \geq (1-\epsilon) \cdot \size{A_n}$,

\item $C(A_n \mid n) \leq O[p^{n} + \log (1/\epsilon)]$, and

\item $\size{A_n} = \Omega(2^{ad(n)})$.
\end{enumerate}
In case $\phi$ has the Kolmogorov property, we can replace \textsc{(iii)} with $\size{A_n} = \Omega(\size{I_n})$.
\end{lemma}

\begin{proof}
By Lemma~\ref{l:polyintervalbinary}, we already know that a small fraction of programs in $I_n$ belong to $M_\phi$, however we wish to obtain a much higher density of minimal indices.  We whittle down the interval $I_n$ so that in the end we are left with a subset $A_n \subseteq I_n$  where at least $(1-\epsilon)\size{A_n}$ elements belong to $M_\phi$.   In addition to the number~$n$, our elimination process will use $O[ap^{n} + \log 1/\epsilon)]$ many bits of non-uniform advice, which will imply that $C(A_n \mid n) = O[p^{n} + \log (1/\epsilon)]$.

Let us see how to obtain $A_n$ with the above properties.  We let
\begin{eqnarray*}
X &=&\text{$M_\phi \cap I_n$,}\\
Y &=&\text{programs in $I_n$ that halt only on input $0$ but are not minimal, and}\\
Z &=&\text{programs in $\bigcup_{i \leq n} I_i$ that halt on input $0$ and at least one other input.}
\end{eqnarray*}
Note that the sizes of $X$, $Y$, and $Z$ are less than $2^m$, where $m =ad(n+1)+1$, so we can write the size of each of these sets in binary on exactly $m$ bits. Let $c \geq ap^{n} + 1$ be a function of $n$ and $\epsilon$ that will be specified later, and let $r = m-c$. Let $t(X), t(Y), t(Z)$ be the truncations of $\size{X}, \size{Y}, \size{Z}$ obtained by retaining the first $c$ most significants bits of the respective binary representation and filling the remaining $r$ bits with $0$'s. Since$\size{X} \geq 2^{-ap^{n}}\size{I_n}$ and  $c \geq ap^{n}+1$, $t(X)$ is not $0$ (on the other hand, $t(Y)$ and $t(Z)$ may be $0$).

Given $n$, the values of $t(X), t(Y)$, and $t(Z)$ can be represented using $3c$ bits, and we assume that we are given this non-uniform information.  Next we build in order the sets $Z', Y'$, and $X'$, which ideally should be $Z, Y$ and respectively $X$, but in fact are just approximations of these sets.

\begin{description}
\item{{\it Step 1} (Construction of $Z'$):}

We enumerate $t(Z)$ elements of $Z$, and these elements make the set $Z'$.

Note that $Z' \subseteq Z$ and there are at most $2^{r+1}$ elements in $Z - Z'$.

\item{{\it Step 2}  (Construction of $Y'$):}

We enumerate $t(Y)$ elements $e$ of $I_n - Z'$ such that
\begin{itemize}
\item $\phi_e(0)$ halts, and
\item $\exists j < e, j \not \in Z'$ such that $\phi_j(0) = \phi_e(0)$.
\end{itemize}

The enumerated elements make the set $Y'$.

If $Z'$ would be exactly equal to $Z$, then $Y'$ would be a subset of $Y$. However, since $Z'$ is not necessarily $Z$, two kind of mistakes can happen.

\begin{itemize}
\item The first kind of mistake occurs when the enumerated index $e$ belongs to $Z - Z'$.  This type of mistake can happen for at most
$2^{r+1}$ elements.

\item The second type of mistake occurs when the $j$ that witnesses the second requirement belongs to $Z - Z'$.  However such a $j$ can cause
at most one program $e$ to be incorrectly considered of type $Y$, when in fact it is of type $X$.
\end{itemize}

Thus in $Y'$, we have:
\begin{itemize}
\item at least $t(Y) - 2 \cdot 2^{r+1}$ programs of type $Y$, 
\item at most $2^{r+1}$ programs of type $Z$ with the first kind of mistake, and
\item at most $2^{r+1}$ programs of type $X$  with the second kind of mistake.
\end{itemize}

\item{{\it Step 3}  (Construction of $X'$):}

We obtain the set $X'$ by enumerating the first $t(X) - 2^{r+1}$ programs that halt on input~$0$ and are not in $Z'$ or in $Y'$.  First note that all programs in~$X$, except those that entered $Y'$ by a second type of mistake, compete in this enumeration.  Therefore the enumeration will eventually collect $t(X) - 2^{r+1}$ elements.  Again if $Z'$ and $Y'$ would be exactly $Z$ and respectively $Y$, then all the enumerated elements would be from $X$.
Since this is not necessarily the case, two types of  mistakes  may happen.

\begin{itemize}
\item The first type of mistake occurs when a program in $Z-Z'$ is enumerated into $X'$. There are at most $2^{r+1}$ such mistakes.

\item  The second type of mistake occurs when a program in $Y - Y'$ is enumerated into $X'$. There are at most $2^{r+1}$ such mistakes.
\end{itemize}
So $X'$ has $t(X) - 2^{r+1}$ elements and except for at most $2 \cdot  2^{r+1}$ many elements, all of the elements in $X'$ belong to $X$.
\end{description}

Note that
\[
\frac{\size{I_n}}{2^m} = \frac{2^{ad(n+1)} - 2^{ad(n)}}{2^{ad(n+1)+1}} = \frac{2^{ap^{n}} - 1}{2^{ap^{n}+1}},
\]
whence by Lemma~\ref{l:polyintervalbinary},
\[
\size{X} 
\geq 2^{-ap^{n}} \cdot \size{I_n} 
= 2^m \cdot \frac{2^{ap^{n}}-1}{2^{ap^{n}}} \cdot \frac{1}{2^{ap^{n}+1}} 
\geq \frac{2^m}{2^{ap^{n}+2}}.
\]
We take $A_n$ to be $X'$, so
\[
\size{M_\phi \cap A_n}  
\geq \size{X'} - 2 \cdot 2^{r+1} 
\geq t(X) - 3 \cdot  2^{r+1},
\]
and therefore
\begin{multline*}
\frac{\size{M_\phi \cap A_n}}{\size{A_n}}  
 \geq 1 - \frac{2^{r+2}}{t(X) - 2^{r+1}}
 \geq 1 - \frac{2^{m - c+2}}{\size{X}}\\
 \geq 1 - \frac{2^{m - c+2}}{2^{m} \cdot 2^{-(ap^{n}+2)}}
 = 1- \frac{1}{2^{c-ap^{n}-4}}.
\end{multline*}
For $c= \lceil \log (1/\epsilon) \rceil + ap^{n} + 5$, the last term is greater than $1-\epsilon$.  Since $A_n$ can be constructed from $n$ and the $3c$ bits that encode $t(X), t(Y), t(Z)$, it follows that $C(A_n \mid n) \leq 3 \log (1/\epsilon) + 3ap^{n} + O(1)$.  Finally,
\begin{multline*}
\size{A_n}  =  t(X) - 2^{r+1}
 \geq (\size{X}-2^{r+1}) - 2^{r+1} \\
 \geq 2^{m-ap^{n+1}-2} - 2^{(m-c)+2}
 \geq 2^m \cdot 2^{-ap^{n}-3}
 = \Omega(2^{ad(n)}).
\end{multline*}
When $\phi$ is a numbering with the Kolmogorov property, $p=1$ and therefore $2^{-ap^{n}}$ is a constant. In this case the last inequality implies $A_n = \Omega(2^m) = O(\size{I_n})$.
\end{proof}

\begin{remark}
The complexity of $A_n$ above depends on $n$ but is small compared to the complexity of most large subsets of $I_n$. Indeed $I_n$ is an interval of integers of the form $\{M, \dotsc, M^{O(p)}\}$, where $M =2^{ad(n)}+1$ and the size of $A_n$ is at least $M$.  The complexity of $A_n$ conditioned by $n$ is $O(\log M + \log 1/\epsilon)$, while the conditional complexity of most subsets of $I_n$ of size $M$ is $\log {M^{O(p)}\choose M} = O(M \log M)$. Furthermore, $p=1$ for numberings with Kolmogorov property, and so in this case the complexity of $A_n$ conditioned by $n$ is a linear function of $\log(1/\epsilon)$.
\end{remark}

\subsection{The champion method}
\label{s:champion}

The following lemma, used in the contrapositive form, provides a sufficient criterion for a set to avoid being $(1,k)$-recursive.  While we give a direct proof in the ``warm-up'' of Section~\ref{sec:wgo}, the argument below goes by contradiction.  When interpreting the Champions~Lemma below, it is useful to keep in mind that a $(\lfloor k/2 \rfloor +1, k)$-recursive set is already recursive \cite{Tra55}.  Epstein and Levin recently discovered a related property which suffices to guarantee that sets contain elements with low Kolmogorov complexity \cite{EL12, She12}.
\begin{championlemma} \label{lem: cl-new}
Let $M$ be a set of binary strings, let $k$ and $m$ be positive integers.  Suppose that $M$ is $(m,k)$-recursive.  Then for all sufficiently large finite sets~$A$ satisfying
\begin{equation} \label{eqn: dc}
\size{M \cap A} \geq \left(1 - \frac{1}{k!(k-m+2)}\right) \cdot \size{A},
\end{equation}
there exists $x \in M \cap A$ with $C(x) \leq C(A) + O(1)$.
\end{championlemma}

\begin{proof}
Let~$M, m$ and $k$ be as in the hypothesis of the theorem, and let~$f$ be the computable  witnessing that $M$ is $(m,k)$-recursive.  Assume towards a contradiction that there exists a sufficiently large finite set $A$  satisfying the density condition~\eqref{eqn: dc} such that for all $x \in M \cap A$ we have $C(x) > C(A) + c$, where $c$ is a constant that will be specified later.  We say that  strings in $M \cap A$ are \emph{high} and strings in $A - M$ are \emph{low}.  Let us assume that $f$ maps $k$-tuples of integers to $\{\H,\L\}^k$ where the label $\H$ asserts that the string is high, and $\L$ asserts that the string is low.  To obtain a contradiction, it suffices to show that $f$ mislabels at least $k-(m-1)$ positions for some vector in $A^k$.

We first present a sketch of the proof.  Most $k$-tuples in $A^k$ consist only of high strings, and in such vectors, $f$ must label at least one position with ``$\H$.''  Henceforth, we restrict $f$ to $A^k$. We count for each string in $A$ how many $\H$ labels it receives among all tuples, with the provision that if a string receives several $\H$ tuples in the same tuple then we count only one.  Consider the lexicographically least string with the largest count of $\H$ among $k$-tuples in $A^k$. We call this string the \emph{champion}. The champion actually has low complexity and consequently, by our assumption, is not in $M \cap A$ despite the fact that there are many $k$-tuples in which it is labeled $\H$.

Let us focus on those $k$-tuples in $A^k$ where $f$ incorrectly labels the champion with $\H$.  Call this set~$E$, and mark in each tuple in $E$ one position where the champion is labeled $\H$.  We can find many such $k$-tuples in~$E$ where all the unmarked positions contain high strings.  Each such vector must have at least one $\H$ label other than the marked one, as $f$ labels at least one position correctly in each $k$-tuple. We count labels the same way as before among vectors in $E$ after ignoring the marked positions. The lexicographically least string with  the largest new count of $\H$ among vectors in $E$ is called the \emph{second place champion}.  Like the champion, the second place champion also has low complexity and therefore, by our assumption, is not in $M \cap A$.  There are still $k$-tuples where both the champion and second-place champion receive incorrectly the label $\H$ and the remaining $k-2$ unmarked positions contain high strings. By iterating this process we obtain a $k$-tuple where~$f$ mislabels each of $k - (m-1)$ positions occupied by champions with $\H$, contradicting that~$f$ always gets at least~$m$ labels correct.

We proceed with the details.  We construct sequentially some strings $x_1, x_2, ..., x_{k-(m-1)}$, the champions mentioned above, such that
\begin{enumerate}[(a)]
\item Each of the $k-(m-1)$ champions has complexity $C(x) < C(A) + c$, for a constant $c$ that will be specified shortly, and
\item for all $1\leq \ell \leq k-(m-1)$, the number of $k$-tuples containing in $\ell$ positions the champions $x_1, \ldots, x_\ell$  labeled $\H$ is at least $[1- \ell /(k-m+2)] \size{A}^{k-\ell}$.
\end{enumerate}
This gives us the contradiction described in the proof sketch because, at stage $k - (m-1)$,  it follows from (b) that there is at least one $k$-tuple containing $k-m+1$ positions occupied in some order by the champions $x_1, \ldots, x_{k-(m-1)}$ which are all labeled $\H$ even though all champions have low complexity and, consequently, taking into account our assumption, they cannot be in $M$.

To start with, let us assume that we can carry out the $k-(m-1)$ iterations of the construction of champions, and let $x_1, x_2, \dotsc, x_{k-(m-1)}$ denote the champion, second place champion, and so on.  Since champion $x_\ell$ can be computed when the set $A$, the index $\ell$ and the function $f$ are given, it follows that for some constant $c$
\[
C(x_\ell) \leq C(A) + \log \ell + 2\log \log \ell + O(1) < C(A) + c
\]
for all $\ell \leq k-(m-1)$, where the ``$2\log \log \ell$'' bits are used to form a delimiter for encoding the pair $\pair{A, \ell}$.  Hence, condition (a) holds.

It remains to demonstrate that all $k-(m-1)$-place champions exist, and that condition (b) is also true.  We will implicitly assume that $\size{A}$ is sufficiently large to complete all $k-(m-1)$ stages of the construction.  Let $\epsilon = 1/[k!(k-m+2)]$. 
By the Bernoulli Inequality, there are at least
\[
(1-\epsilon)^k  \size{A}^k \geq (1-k\epsilon) \size{A}^k \geq \left(1 - \frac{1}{k-m+2}\right) \size{A}^k
\]
$k$-tuples in $A^k$ consisting of only high strings, and all these tuples contain at least one $\H$ in their label.  Consequently, the first champion can be chosen as stated in the proof sketch. Also, the number of tuples containing $x_1$ labeled $\H$ is at least $(1- 1/(k-m+2))\size{A}^{k-1}$.

We shall construct $x_{\ell+1}$ by induction.  Assume that $x_1, \ldots, x_\ell$ exist, and that there are at least $[1- \ell/(k-m+2)]\size{A}^{k -\ell}$  $k$-tuples in $A_k$ in which $x_1, \ldots, x_\ell$ appear and each receive the label $\H$. Let $E$ be the set of such $k$-tuples, and  let $W$ be those vectors in $E$ for which all the non-champion positions contain high strings.  Let $B$ be the complement of $W$ in $E$, that is the set of all $k$-tuples for which $f$ labels all the current champions $x_1, \ldots, x_\ell$ with $\H$ and there is at least one low string in the remaining positions.  We would like to find an upper upper bound for $B$, so let us set $B'$ to be all vectors in $A^k$ that contain the current champions, regardless of how they are labeled, and at least one additional low string. Then $B$ is included in $B'$, so $\size{B} \leq \size{B'}$.

Let $D$ be the set of vectors in $A^k$ which contain $x_1, \dotsc, x_\ell$ and have high strings in the remaining $(k-\ell)$ positions.   Let $d$ be the number of possible ways of placing $x_1, \dotsc, x_\ell$ into a $k$-tuple.  Now $d \leq {k \choose \ell} \ell!$, and equality holds if the champions are all distinct.  Furthermore,
\begin{align*}
\size{D} \geq d (1-\epsilon)^{k-\ell}\size{A}^{k-\ell} \geq [1-(k-\ell)\epsilon] d \size{A}^{k-\ell}
\end{align*}
because the current champions $x_1, \ldots, x_\ell$ are all low and thus we do not overcount  when we consider different positions for the $\ell$ current champions.  Now $B'$ is the set of all $k$-tuples containing $x_1, \ldots, x_\ell$  except for those in $D$, and thus
\[
\frac{\size{B'}}{d \size{A}^{k-\ell}} \leq 1 - [1-(k-\ell)\epsilon] =  (k-\ell) \epsilon.
\]

Finally, we bound the size of $W = E \setminus B$.  Note that
\begin{eqnarray*}
\size{E} - \size{B} \geq \size{E} - \size{B'}  \geq \left[1 - \frac{\ell}{k-m+2} - (k-\ell) d \epsilon\right] \size{A}^{k-\ell},\
\end{eqnarray*}
and applying our estimate for $d$, we obtain
\[
\size{W}
 \geq \left[1 - \frac{\ell}{k-m+2} -  k (k-1) \dotsc (k-\ell) \epsilon \right] \size{A}^{k-\ell} 
 \geq \left(1 - \frac{\ell+1}{k-m+2}\right) \size{A}^{k-\ell}.
\]
Therefore the $\ell+1$-place champion can be chosen as described in the proof sketch, and the number of tuples containing $x_1, \ldots, x_{\ell+1}$ all labeled $\H$ is at least $[1 - (\ell+1)/(k-m+2)] \size{A}^{k-(\ell+1)}$.
\end{proof}

We show next that the converse fails.

\begin{proposition}
For every positive integer $k$, there exists a set of natural numbers $M$ that is not $(1,k)$-recursive and yet it  satisfies the conditions of the Champions Lemma, i.e., for all finite sets of natural numbers $A$ at least one of the following holds true:
\begin{enumerate}[\scshape (i)]
\item $\size{M \cap A} < \left(1 - \frac{1}{(k+1)!}\right) \cdot \size{A}$, or
\item there exists $x \in M \cap A$ such that $C(x) <  C(A) + O(1)$.
\end{enumerate}
\end{proposition}

\begin{proof}
Let us fix $k$ and let $\epsilon = 1/(k+1)!$. We construct a set $M$ that satisfies for all integers $n$ and for all finite sets $A$ the following requirements:

\begin{description}
\item{$R_n$:} The $n$-th computable function $f_n$ gets all labels for M wrong on some $k$-tuple, and

\item{$S_A$:} $A$ satisfies either \textsc{(i)} or \textsc{(ii)} above.
\end{description}
To construct $M$ we partition $\N$ into consecutive intervals $I_0, J_0, I_1, J_1, \dotsc,\linebreak[0] I_n, J_n, \dotsc$.  Each $I_n$ is an interval with exactly $k$ elements which is used to diagonalize against $f_n$ in the obvious way: $f_n$ assigns labels to the elements of $I_n$, and we define $M$ in the opposite way so that all labels are incorrect. This ensures that $M$ is not $(1,k)$-recursive.

We now define the intervals $J_n$ to satisfy the requirements $S_A$.  These intervals will satisfy $M \cap J_n = \emptyset$ for all $n$.  More precisely, $J_n$ will witness that $S_A$ is satisfied for the finite sets $A$ where $|A| > k/(1-\epsilon)$ and $n$ is the least index such that $A \cap I_n \neq  \emptyset$.  We say that such sets $A$ form the \emph{target} of $J_n$.

If $t=\max I_n$, we take $J_n = \{t+1, \ldots, t+m\}$ where $m$ is large enough so that for all $A$ in the target of $J_n$, if $A$ contains an index greater than $t+m$ then
\[
C(A) > \max \{C(x) \colon x \in I_0 \cup \ldots \cup I_n\}.
\]
On the other hand, if $A$ is in the target of $J_n$ and $\max A \leq t +m$, then
\[
\size{M \cap A} \leq \size{I_n \cap A} \leq k < (1-\epsilon)\size{A}.
\]
Thus $S_A$ is indeed satisfied for all $A$ in the target of $J_n$. 

It only remains to satisfy $S_A$ for sets $A$ with $\size{A} \leq k/(1-\epsilon)$. This is easy because any element $x$ in such a set can be described with $C(A) + O(1)$ bits, where the $O(1)$ bits are used to represent the rank of $x$ in some canonical representation of $A$.
\end{proof}

We present the following user-friendly version of the Champions Lemma.

\begin{theorem} \label{thm:nonapprox}
Let $M$ be a set of binary strings.  Suppose there exist  a sequence of distinct finite sets $A_0, A_1, \dotsc$ and a sequence of positive reals $\epsilon_0, \epsilon_1 \dotsc$ with limit 0 such that
\begin{enumerate}[\scshape (i)]
\item $|M \cap A_n| \geq \left(1 - \epsilon_n\right) \cdot \size{A_n}$, and
\item for all $x \in M \cap A_n$, $C(x) \geq C(A_n) + \log [C(A_n)]$.
\end{enumerate}
Then $M$ is not $(1,k)$-recursive for any $k$.
\end{theorem}

\begin{proof}
Choose $\epsilon_k = 1/[k!(k+1)]$, and apply the Champions~Lemma in the contrapositive form.  Note that the ``$O(1)$'' term in the Champions~Lemma depends of the constant $k$, but since $\limsup \size{A_n} = \infty$, we obtain for each~$k$ an infinite sequence of finite sets, each with the requisite density and complexity.
\end{proof}

\section{Non-approximability of $\MIN_\phi$}
\label{s:minset}

We first show that the set of Kolmogorov random strings, $\{x \colon C(x) \geq \size{x}\}$ is not $(1,k)$-recursive for any~$k$.
For an arbitrary function $g$, let $\HIGH_g = \{x : C(x) \geq g(\size{x})\}$.
\begin{theorem} \label{t:random}  \label{thm: howfatisshe}
Let $g$ be a computable function such that both $g(n)$ and $n-g(n)$ are unbounded, and let $c$ be a nonnegative constant.  Then
for any positive integer $k$, 
\begin{enumerate}[\scshape (i)]
\item $\HIGH_{g}$ is not $(1,k)$-recursive, and

\item $\HIGH_{n-c}$ is not $(1,k)$-recursive.
\end{enumerate}
\end{theorem}
\begin{proof}
(\textsc{i}): We take $A_n$ to be set of strings of length $n$. It holds, by counting the maximum possible number of relevant descriptions, that $\size{\HIGH_g \cap A_n} \geq (1- 2^{-(n-g(n)-1)})\size{A_n}$. The complexity $C(A_n)$ is bounded by $C(n) + O(1)$, which for infinitely many $n$ is at most $g(n)/2$ because no unbounded, computable function is a lower bound for Kolomogorov complexity \cite{LV08, ZL70}. Finally, for every string $x$ in $\HIGH_g \cap A_n$, we have
\[
C(x) \geq g(n) > g(n)/2 + \log [g(n)/2].
\]
The conditions of Theorem~\ref{thm:nonapprox} are satisfied for $M = \HIGH_g$, and the conclusion follows.

\medskip

(\textsc{ii}): Since the set of $n$-bit strings not in $\HIGH_{n-c}$ do not form a vanishing fraction of all $n$-bit strings, part~(\textsc{ii}) requires a more elaborate analysis. Let us focus on the case $c=0$, that is, the case of \emph{random strings}.  For positive values of $c$, the proof is similar and slightly easier.

Let $I_n$ denote the set of $2^n$ binary strings of length~$n$.  First we argue that for infinitely many lengths~$n$, at least $2^{n-2}$ strings of length~$n$ are random.  Indeed, if there were less than $2^{n-2}$ random strings of length~$n$, then at most $(2^n - 1) - (2^n - 2^{n-2}) < 2^{n-2}$ programs of length less than~$n$ could describe strings of length greater than~$n$, and the number of programs of length~$n$ that describe strings of length greater than~$n$ is at most $2^n$. Thus at least $2^{n+1} - 2^n - 2^{n-2} > 2^{n-1}$ strings of length $2^{n+1}$ are random.

Let $R$ denote the set of random strings and let $t_n$ be the number of nonrandom strings in $I_n$ encoded in binary with the last $2 \log n$ bits replaced with ``$00\dotsc 0$,'' and let $T_n$ be the first $t_n$ elements in $I_n$ found to be nonrandom.  We throw these strings out of the ``arena'' $I_n$, and what's left are mostly random strings.  Let $A_n = I_n \setminus T_n$.  Then  at most $2^{2\log n + 1} = 2n^2$ in $I_n \setminus T_n$ are nonrandom. For any $\epsilon > 0$ and all sufficiently large $n$, we have $2^{n-2} - 2n^2 \geq (1-\epsilon) 2^{n-2}$, and for those infinitely many~$n$ which are both sufficiently large and for which there are at least $2^{n-2}$ random strings of length~$n$, we have
\[
\size{R \cap A_n} \geq 2^{n-2} - 2n^2 \geq (1-\epsilon) 2^{n-2} \geq (1-\epsilon) \size{A_n},
\]
which satisfies condition~{\scshape (i)} of Theorem~\ref{thm:nonapprox}.  Furthermore,
\[
C(A_n) \leq C(T_n) + O(1) \leq n - 2 \log n + O(1),
\]
so we satisfy condition~{\scshape (ii)} as well because every string in $A_n$ is random.  The theorem follows. \phantom\qed
\end{proof}

\begin{remark}
Regardless of which underlying universal machine is used to measure Kolmogorov complexity, Theorem~\ref{thm: howfatisshe}~\textsc{(ii)} holds for at most finitely many negative values~$c$.  This follows from the fact that any string of length~$n$ can described using $n+O(1)$ bits.  
\end{remark}

\begin{theorem}
\label{t:onekrec}
For all numberings $\phi$ with Kolmogorov property, for all $k$, $\MIN_\phi$  is not $(1,k)$-recursive. 
\end{theorem}
\begin{proof}
Fix $k$, and let $\epsilon = 1/[k!(k-1)]$.  Apply Lemma~\ref{lem:pdl} with $p=1$ (corresponding to numberings with Kolmogorov property) to obtain sets $A_0, A_1, \dotsc$ satisfying

\begin{enumerate}[(1)]
\item $\size{M_\phi \cap A_n} \geq (1-\epsilon) \cdot \size{A_n}$,
\item $C(A_n) \leq \log n + O[\log (1/\epsilon)]$, and 
\item $\size{A_n} = \Omega(2^{\Omega(n)})$.  
\end{enumerate}
Then $\size{\MIN_\phi \cap A_n} \geq (1-\epsilon) \cdot \size{A_n}$, and by Lemma~\ref{l:highmin} every $x \in \MIN_\phi \cap A_n$ satisfies
\[
C(x) \geq \size{x} - \log \size{x} = \Omega(n)  \geq C(A_n) + \log [C(A_n)].
\]
It follows from Theorem~\ref{thm:nonapprox} that $\MIN_\phi$ is not $(1,k)$-recursive.
\end{proof}

\begin{remark}
The proof technique used in Theorem~\ref{t:onekrec} cannot be used to extend this result to acceptable numberings.  The reason is that the argument in the density-boosting Lemma~\ref{lem:pdl} also goes through for the one-dimensional version of $\MIN_\psi$,  ${\rm SD}_\psi=\{e : (\forall j <e)\: [\psi_j(0) \neq \psi_e(0)]\}$, which is known to be $(1,2)$-recursive for some acceptable numbering~$\psi$ (care of the Remark following Theorem~2.3 in \cite{ST12}).
\end{remark}

By modifying the construction in \cite[Theorem~2.3]{ST12}, one can show that  ${\rm SD}_\psi$ is  $(1,2)$-recursive for some polynomially-bounded numbering $\psi$, which implies  that the method in Theorem~\ref{t:onekrec} does not even extend to polynomially-bounded numberings.  The necessary modification  in \cite[Theorem~2.3]{ST12} is to use intervals $J_{n,n \log^3 n}, \dotsc, J_{n,1}$ rather than $J_{n,2^{n+1}}, \dotsc, J_{n,1}$ to code the first $\log n + 3 \log \log n$ of Chaitin's $\Omega$ rather than the first~$n$ bits.  Then the size of the interval $I_n$ becomes less than $n^2 \log^6 n$, and hence the index $e_n$ is at most $\sum_{k=1}^n k^2 \log ^6 k = O(n^4)$.  Then the numbering $\psi$ constructed becomes polynomially bounded, and the Kolmgorov complexity argument at the end still works.

\section{Non-approximability of $\min_\phi$}
\label{s:minf}
The main focus in this section is on the \emph{shortlist for functions problem}, Question~\ref{ques: poly-log list}, but first we consider a different type of approximation.
 We say that a function $f$ is \emph{($K$-)approximable from above} if there exists a uniform sequence of ($K$-)computable functions $f_0, f_1, \dotsc$ such that for all $x$, $f_s (x) \geq f_{s+1} (x)$ and $\lim_s f_s(x) = f(x)$.  We define \emph{($K$-)approximable from below} similarly, but with the inequalities reversed. In some sense, $\min_\phi (e)$ is the function analog of Kolmogorov complexity for strings, $C(x)$.  We investigate whether $\min_\phi(e)$ has approximability properties similar to those of $C(x)$.  $C(x)$ is approximable from above, but not from below~\cite{LV08}. For $\min_\phi(e)$, we have the following contrasting result.

\begin{proposition}
For any numbering $\phi$, the function $\min_\phi$ is $K$-approximable from below.  If $\phi$ is computably bounded, then $\min_\phi$ is not $K$-approximable from above.
\end{proposition}
\begin{proof}
Let $\phi$ be a numbering.  Using a $K$-oracle, one can enumerate all pairs $\pair{e,j}$ such that $\phi_e \neq \phi_j$ by searching for the least input on which either $\phi_e$ and $\phi_j$ disagree or where one function converges and the other one doesn't.  Define $f_s(e)$ to be the the least index~$j$ such that $\phi_j(x) = \phi_e(x)$ for all $x \leq s$.  Now $f_s(e)$ is an increasing function which eventually settles on $\min_\phi(e)$, and therefore $f_0, f_1, \dotsc$ is a uniform, $K$-computable sequence of functions witnessing that $\min_\phi$ is $K$-approximable from below.

If $\phi$ were computably bounded and also approximable from above, this would imply that $\min_\phi$ is a $K$-computable function, whence $\MIN_\phi \leq \emptyset'$, contradicting Lemma~\ref{lem: nolala}.
\end{proof}

Although we can approximate $\min_\phi$  from below using a halting set oracle, no unbounded computable function bounds $\min_\phi$ from below when $\phi$ is computably bounded.  In this sense, $\min_\phi$ resembles Kolmogorov complexity $C(x)$ \cite{LV08, ZL70}.
\begin{proposition}
For every computably bounded numbering  $\phi$, there exists no unbounded computable function $h$ such that $\min_\phi(e) \geq h(e)$ for all $e$. 
\end{proposition}
\begin{proof}
Suppose that for some computable, unbounded $h$,  $\min_\phi(e) \geq h(e)$, for all $e$. Let $U$ be the underlying universal machine for the Kolmogorov complexity function $C$.  Define a numbering $\psi$ by
\[
\psi_q = \phi_{U(q)},
\]
and let $t$ be a ``translator" program such that $\phi_{t(q)} = \psi_{q}$ and $s$ a computable function such that $t(q) \leq s(q)$ for all $q$.  Let $s'$ be the unbounded computable function defined as follows: $s'(x)$ is the smallest $y$ such that $s(y) \geq x$.  Let $e$ be an arbitrary index, and let $p$ be a shortest program such that $U(p)=e$. Thus $\size{p} = C(e)$ and therefore $p < 2^{C(e)}$.  Note that $\phi_{t(p)} = \psi_p = \phi_{U(p)} = \phi_e$ and consequently $t(p) \geq \min_\phi(e) \geq h(e)$. Hence, $s(p) \geq h(e)$ and thus $s(2^{C(e)}) \geq h(e)$. It follows that $C(e) \geq \log (s'[h(e)])$ for all $e$.  Thus $C(e)$ is lower-bounded by a computable, unbounded function, which is impossible~\cite{LV08, ZL70}.
\end{proof}

We now turn our attention to the \emph{shortlist for functions problem}. The next theorem and proposition show that in the general case of acceptable numberings, only a weak lower bound on the length of shortlists is possible.

\begin{theorem}
 \label{thm:listlength2}
For any computably bounded numbering $\phi$ and any constant $k$, there is no computable function~$f: \N \to \N^k $ such that $\min_\phi(e) \in f(e)$ for all $e$.
\end{theorem}

\begin{proof}
Fix a computably bounded numbering $\phi$, and define $e(n)$ to be the first index found such that that $\phi_{e(n)}(0) = n$.  Suppose that $k$ is the least positive integer for which some computable function $f: \N \to \N^k$ satisfies $\min_\phi[e(n)] \in f[e(n)]$ for all $n$.  If $k=1$, this immediately contradicts Lemma~\ref{lem:immunity}, so we may assume $k \geq 2$.  Observe that we may have $\min_\phi(x) \notin f(x)$ for indices $x$ which do not equal $e(n)$ for some $n$.

Let $m(n) = \min f[e(n)]$, and define $A = \{m(n) \colon n \geq 0\}$.  There are two cases.
\begin{description}
\item{\it Case~1}: $A$ is finite.

Then $\min_\phi [e(n)] \notin A$ for all but finitely many~$n$.  But then we could find a computable function $f'$ mapping indices to lists of length $k-1$ which hardcodes a correct answer for these finitely many values and maps $f'(x) = f(x) \setminus \min [f(x)]$ elsewhere.  Since $\min_\phi[e(n)] \in f'[e(n)]$ for all~$n$, this contradicts the minimality of $k$.

\item{\it Case~2}: $A$ is infinite.

Let $g$ be the order from Lemma~\ref{lem:okra} and let
\[
t(n) = \min \{x \colon g(x) \geq n\}.
\]
Then $g[t(n)] \geq n$ for all $n$.  Also let $j(n)$ be the smallest index such that $m[j(n)] \geq t(n)$. Note that since $A$ is infinite, $j(n)$ is defined for every $n$.   Now the $k$-tuple $f(e[j(n)]) = (y_1, \ldots, y_k)$ contains the minimal index $y_i = \min_\phi (e[j(n)])$. Since $m[j(n)]$ is the smallest element in the $k$-tuple, it holds that $y_i \geq m[j(n)]$.  By Lemma~\ref{lem:okra}, $C(y_i) \geq g(y_i)$.  Thus,
\begin{equation}
\label{e:one}
C(y_i) \geq g(y_i) \geq g(m[j(n)]) \geq g[t(n)] \geq n.
\end{equation}
On the other hand,  since $y_i$ is an element of the $k$-tuple and $e[j(n)]$ can be computed from $n$, it follows that
\begin{equation}
\label{e:two}
C(y_i) < \log n + 2 \log k + O(1).
\end{equation}
For large enough $n$, the inequalities~\eqref{e:one} and~\eqref{e:two} contradict each other.
\end{description}
Therefore no such $k$ exists.
\end{proof}

The next proposition shows that Theorem~\ref{thm:listlength2} is essentially optimal.
\begin{proposition}
For any computable order~$g$, there exists an acceptable numbering~$\psi$ and a computable function~$f$ which maps each index~$e$ to a list of size at most~$g(e)$ such that $\min_\psi(e) \in f(e)$.
\end{proposition}

\begin{proof}
Let $\phi$ be an acceptable numbering with $\phi_0$ and $\phi_1$ both being the everywhere divergent function.  For $n \geq 1$, let $a_n$ denote the $n^\text{th}$ smallest positive integer satisfying $g(a_n) > g(a_n - 1)$, and define the numbering $\psi$ by
\[
\psi_e = 
\begin{cases}
\phi_n &\text{if $e = a_n$ and $n \geq 2$,}\\
\phi_0 &\text{otherwise},
\end{cases}
\]
and let
\[
f(e) = \{a_n : a_n \leq e \text{ and } n \geq 2\} \cup \{0\}.
\]
Now $f(e)$ contains all the $\psi$-minimal indices up to $e$ and has size at most $g(e)$.
\end{proof}

For numberings with the Kolmogorov property, a sharper lower bound is possible.

\begin{theorem}
\label{t:quadbound}
Let $\phi$ be a numbering with the Kolmogorov property and let $f$ be a computable function which maps each index $x$ to a list of indices containing $\min_\phi(x)$.  Then $\size{f(x)} = \Omega(\log^2 x)$ for infnitely many~$x$.
\end{theorem}

\begin{proof}
Let $\phi$ and $f$ be as in the hypothesis, and let $e(x)$ be the computable function which outputs the first index found such that $\phi_{e(x)}(0) = x$.  Let $U$ be the universal machine for Kolmogorov complexity, and define a further numbering~$\psi$ by $\psi_p = \phi_{e[U(p)]}$ if $U(p) \converge$ and $\psi_p$ being the everywhere divergent function otherwise.  Since $U$ is an optimal machine, there exists a function~$t$ such that $U [t(z)] = \phi_z(0)$ and $t(z) \leq O(z)$.  Define a computable function $g$ from indices to sets of descriptions for $U$ by
\[
g(x) = \{ t(z) \colon z \in f[e(x)]\}.
\]
By the Kolmogorov property, there exists a linearly bounded (but not necessarily computable) function $h$ such that $\psi_p = \phi_{h(p)}$.  Now observe that  whenever $U(p) = x$, we have
\[
\phi_{h(p)} = \psi_p = \phi_{e[U(p)]} = \phi_{e(x)},
\]
whence
\[
\min_\phi [e(x)] \leq h(p) \leq O(p).
\]

Fix an $x$, and let $p$ be the least program such that $U(p)=x$.  Since $\min_\phi[e(x)] \in f[e(x)]$, there exists $q \in g(x)$ such that $U(q) = \phi_{e(x)}(0) = x$ and $q \leq O(\min_\phi[e(x)]) \leq O(p)$.  So $g(x)$ is a list with the same length as~$f(x)$ containing a description for $x$ which is only a constant many bits longer than the minimal $U$-description for~$x$.  By \cite[Theorem~I.3]{BMVZ13}, $g(x)$, and hence $f(x)$ as well, must have length $\Omega(\log^2 x)$ for infinitely many~$x$.
\end{proof}

Our final result shows that some numberings with the Kolmogorov property do not admit shortlists.

\begin{theorem}
\label{t:expbound}
There exists a Kolmogorov numbering $\psi$ such that if $f$ is a computable function which maps each index~$x$ to a list of indices containing $\min_\psi(x)$, then $\size{f(x)} = \Omega(x)$ for infinitely many $x$.
\end{theorem}
\begin{proof}
On a high level, we use an approach from~\cite[Theorem~I.4]{BMVZ13}.   The \emph{total complexity of a string $y$ conditioned by $x$}, as originally defined by Muchnik and used in \cite{BMVZ13,Ver09}, is
\[
T(y \mid x) = \min \{\size{q} \colon U(q,x) = y \text{ and $U(q,z) \converge$ for all $z$}\},
\]
where $U$ is the universal machine for Kolmogorov complexity.  Note that if $\psi$ is some numbering and for all $x$, $\min_\psi(x) \in f(x)$, where $f$ is a computable function, then $T(\min_\psi(x) \mid x) \leq \log\size{f(x)} + O(1)$, for all $x$.  Thus our plan is to define a Kolmogorov numbering $\psi$ such that $T(\min_\psi(x) \mid x) \geq \log x - O(1)$ for infinitely many $x$.  Our Kolmogorov numbering $\psi$ will be based on an arbitrary Kolmogorov numbering $\phi$, and we will use binary strings as inputs to $\psi$ rather than integers.  The inputs with prefix~1 will be used to ensure that $\psi$ is a Kolmogorov numbering, and inputs with prefix~0 will be used to code for strings with high total complexity. 

For every string $x$, we define $\psi_{1^r x} = \phi_x$, where $r = a+c+d+1$ and $a$, $c$ and $d$ are constants that will be specified later.   Next we define $\psi_{0\alpha \beta e}$ for all strings $\alpha$ of length $c$, all strings $\beta$ of length $d$,  and every string $e$ of length $a(n+1)$, for some integer $n$.  For strings $x$ which are not of either of these two forms, we set $\psi_x$ to be the everywhere undefined function.  Our goal is to obtain infinitely many $x$ and $\alpha \beta e$ such that $\min_\psi(1^r x) = 0\alpha \beta  e$ and $T(0\alpha \beta e \mid 1^r x) \geq \size{1^rx} - O(1)$, as then the theorem follows by the discussion in the previous paragraph.

The construction of $\psi$ uses a game which we call $\Gamma_{n,\alpha, \beta }$.  The precise roles of $\alpha$ and $\beta$ will be clarified later, but, in short, they provide non-uniform advice information necessary for satisfying some requirements.  The  game indicates how some of the functions $\psi_{0 \alpha \beta e}$ are calculated on input $0$; on all inputs different than $0$, these functions are undefined. The game is played between two players, Matchmaker and Spoiler.  Roughly speaking, Matchmaker selects pairs $(e,x)$, with the effect that $\psi_{0 \alpha \beta e}$ is set equal to $\phi_x$, whereas the Spoiler checks if $e$ or $\phi_x$ violate requirements which demand that $e$ has high total complexity conditioned by $1^rx$ and that $x$ is in $\MIN_{\phi}$.  When such violations are found, Spoiler blocks the pair $(e,x)$, and  Matchmaker is forced to look for another pair $(e,x)$.

Our analysis of the game $\Gamma_{n, \alpha, \beta}$ consists of both combinatorial and computational components.  In terms of combinatorics, we have to show that the Matchmaker does not run out of unspoiled pairs $(e,x)$ to choose from as moves.  The computational aspect has to do with requirements and appropriate definition for $\psi_{0\alpha \beta e}$.  The following description  isolates the combinatorial aspect of the game.

\paragraph{The game.}  The game is played on a board which is a $k' \times k$ table with $k' \geq k$.  Initially all the cells are unblocked.  Cells are indexed by their row and column, so cell $(e,x)$ is the entry on row $e$ and column $x$.  The Spoiler and Matchmaker take turns making the following types of moves, respectively.
\begin{description}
\item{\it Matchmaker move:} She picks a cell $(e,x)$ that is not blocked and places a pawn on it.  At the same time, all the cells on row $e$ and all the cells on column $x$ are blocked.  She also has the option to pass.

\item{\it Spoiler move:}  The Spoiler has two type of moves: 
\begin{description}
\item{\sl Column move:}  He picks a column $x$ and blocks all the cells in this column; 

\item{\sl Row move:} He picks one cell in each column (not necessarily in the same row, despite the name) and blocks them.
\end{description}
\end{description}
Spoiler is permitted to do at most $k/4$ Column moves and at most $k/16$ Row moves.  Matchmaker wins if at the end there is a pawn on a cell $(e,x)$ that is not blocked.

\begin{lemma}
\label{l:gl}
For every $k$ and $k' \geq k$, Matchmaker has a strategy to win the game on the board of size $k' \times k$. If Spoiler uses
a computable strategy, then Matchmaker has a computable winning strategy.
\end{lemma}
\begin{proof}
Matchmaker's strategy is to place a pawn on the first unblocked cell found and then pass until Spoiler blocks that cell.  If this happens, she places another pawn.  Let us check that she can always place a pawn, which implies that she wins the game.

Note that a Column move blocks $k'$ cells, and a Row move blocks $k$ cells. Thus the Spoiler can block at most $(k/4)k' + (k/16)k$ cells during the entire game.  Since Matchmaker only plays after a Spoiler move, Matchmaker makes at most $(k/4) + (k/16) = (5k)/16$ moves, and in each move she blocks $(k+k'-1)$ cells (a row and a column).  Thus the total number of blocked cells is
\[
\left(\frac{k}{4}\right) k' + \left(\frac{k}{16}\right) k + \left(\frac{5k}{16} \right) (k+k'-1) \leq \frac{15}{16} \cdot k'k,
\]
and therefore there always exists an unblocked cell where Matchmaker can place a pawn.
\end{proof}

Now we are prepared to define outputs $\psi_{0\alpha\beta e}(0)$ using the game $\Gamma_{n,\alpha, \beta}$.  The rules of this game are as in Lemma~\ref{l:gl}, but we need to relate the board parameters, rows, and columns to aspects of the numbering~$\psi$.  From Lemma~\ref{lem:pdl} with $p=1$ (corresponding to numberings with Kolmogorov property), we infer the existence of constants $a, c$, and for every $n$, of sets $A_n$ and intervals $I_n$ such that:
\begin{enumerate}[(i)]
\item  $A_n \subseteq I_n=\{2^{an}+1, \dotsc, 2^{a(n+1)}\}$ (therefore, when converted to binary, the elements of $I_n$ have lengths between $an+1$ and $an+a$),

\item $\size{M_\phi \cap A_n} \geq (7/8) \cdot \size{A_n}$,

\item $C(A_n \mid n) \leq c$, and

\item $\size{A_n} = \Omega(\size{I_n})$.
\end{enumerate}

First we describe the \emph{intended} setting of parameters for the game $\Gamma_{n,\alpha,\beta}$, even though for some strings $\alpha$ and $\beta$, the setting will be slightly different as we will explain.  In the intended setting the game is played on a board where columns are indexed by the strings $x$ in $A_n$, whose binary expansions have length at least $an+1$, and the rows are indexed by the strings $e$ of length $a(n+1)$.  Thus the board has dimensions $k' \times k$, where $k = \size{A_n}$, and $k' = 2^{a(n+1)}$.  The set $A_n$ is not computable, but $C(A_n \mid n) \leq c$ by Property~(iii).  Hence we can use $e$ (from which we can derive $n$) and the string $\alpha$ of length $c$ as nonuniform advice for computing the set $A_n$.  Note that in case $\alpha$ is a correct advice, there exists a constant $\gamma$ such that $k \geq \gamma 2^{a(n+1)}$ by Property~(iv) above.

During the game we would like to determine whether or not an arbitrary $x \in A_n$ belongs to $M_\phi$.  For some $x \in M_\phi$,  $x$ may appear at some stage~$s$ to be in the complement of $M_\phi$, because there is some $y < x$, under the lexicographical ordering of binary strings, such that $\phi_{x,s}(0) \converge =  \phi_{y,s}(0) \converge$, even though at some later stage $t$, $\phi_{y,t}$ converges on some nonzero input.  In this case $y$ has threatened $x$, and this is a situation that we want to avoid.  Formally, a string $y$ is a \emph{threat} to $A_n$ if $y < \max A_n$ and $\phi_y$ converges on more than one input.  Let $T$ be the number of threats of $A_n$. If we knew $T$, we could determine all threats.  While we do not know $T$, using $n$ and a constant number of advice bits we can determine a number $T'$ which is within $k/8$ of $T$.  Here is how.  Since any threat has index less than $\max A_n$, we have $T < 2^{(a+1)n}$.  Let $d = \lceil \log 1/\gamma \rceil + 3$. We write $T$ on exactly $a(n+1)$ bits and we let $T'$ be the number obtained by retaining the $d$ most significant bits in the binary expansion of $T$ and filling the rightmost $a(n+1) - d$ bits with $0$'s. Then
\[
T-T' < 2^{a(n+1)-d} \leq 2^{-d} (k/\gamma) \leq k/8.
\]
In the intended setting of the game $\Gamma_{n,\alpha, \beta}$, $\beta$ is the string consisting of the first $d$ bits in the binary expansion of $T$. 

Games of the form  $\Gamma_{n,\alpha', \beta'}$ with incorrect advice $\alpha'$ or $\beta'$  are also played, and for this reason at the start of the game we check if the size of $A'_n$, the set constructed from $n$ and advice $\alpha'$, is at least $\gamma 2^{a(n+1)}$. If this is not the case, then the game  $\Gamma_{n,\alpha', \beta'}$ is not played.

In the following analysis we assume that $\alpha$ and $\beta$ are correct, and therefore the game $\Gamma_{n,\alpha, \beta}$ has the intended parameters. Before the game starts we construct $A_n$ and, using $n$ and $\beta$, we enumerate threats of $A_n$ until we find $T'$ of them. In this way we find a set $B_n$ containing all threats of $A_n$ except at most $k/8$ of them.

Next we describe a computable Spoiler strategy by indicating the situations in which he plays a Column move and the situations where he plays a Row move.

\begin{description}
\item{ \it Column move.}
Spoiler plays a Column move $x$ at stage $s$ if column $x$ is not already blocked and one of the following happens:
\begin{enumerate}[(a)]
\item some index $y < x$ (lexicographically) is found such that $y \not \in B_n$ and $\phi_{x,s}(0) \converge =  \phi_{y,s}(0) \converge$ (we say that $y$ \emph{attacks} $x$), or

\item some input $u \not= 0$ is found such that $\phi_{x,s}(u) \converge$.
\end{enumerate}
There are two cases to analyze.
\begin{description}
\item{\sl Case 1:} (a) happened and $x \in \MIN_\phi$. In this case,  $y$ must be one of the at most $k/8$ threats that are outside $B_n$.  Since any string $y$ can attack at most one string in $\MIN_\phi$, Case~1 can occur at most $k/8$ times.

\item{\sl Case 2:} (b) happened or $x \not \in \MIN_\phi$. By Property~(ii) of $A_n$,  Case 2 can occur at most $k/8$ times.
\end{description}
Therefore, the number of Column moves is bounded by $k/8 + k/8=k/4$, as required.

\item{\it Row move.}
If a string $q$ of length  at most $\log (\size{A_n}) - 4$ is found such that $U(q, 1^r x)$ halts for all $x \in A_n$, then Spoiler makes a Row move and
blocks all the cells $(e,x)$ such that $U(q,1^r x) = 0\alpha \beta e$. There are at most $k/16$ such moves, as required.
\end{description}
This concludes the description of Spoiler's strategy.

By Lemma~\ref{l:gl}, Matchmaker has a computable winning strategy. This strategy permits us to define the function $\psi_{0\alpha \beta e}$ as follows.  Initially $\psi_{0\alpha \beta e}$ is undefined on all inputs. When Matchmaker puts a pawn on cell $(e,x)$, we set $\psi_{0\alpha \beta e}(0) = \phi_x(0)$.

Because Matchmaker wins the game  $\Gamma_{n,\alpha, \beta}$, at the end of the game, some cell $(e,x)$ has a pawn and is not blocked. We call $(e,x)$, the \emph{winning cell} of the game. We derive that
\begin{enumerate}[(1)]
\item $x \in M_\phi$ (otherwise the cell $(e,x)$ would be blocked by a Column move of the Spoiler),

\item $\psi_{1^rx} = \psi_{0\alpha \beta e}$ (since $\psi_{0\alpha \beta e}(0) = \phi_x(0) = \psi_{1^rx}(0)$ and these functions are only defined on input  $0$), and

\item $T(0 \alpha \beta e \mid 1^r x) > \log (\size{A_n}) - 4 = \size{1^r x} - O(1)$  (otherwise the cell $(e,x)$ would be blocked by a Row move of the Spoiler).
\end{enumerate}

From (1), (2), and the fact that $r = a+ \size{\alpha} + \size{\beta} +1$, $x$ has length at least $an +1$ and $e$ has length $an+a$,  we infer that $\min_\psi(1^rx) \leq 0\alpha \beta e$. Let us analyze the possible situations when the inequality might be strict. It is not possible that $\psi_{1^rx'} = \psi_{1^r x}$ for any $x'<x$, because this would contradict $x \in \MIN_\phi$.  But the inequality may still be strict because it can happen that there exist $\alpha' \not= \alpha$  or $\beta' \not= \beta$ and $e'$ such that $(e',x)$ is the winning cell in the game $\Gamma_{n,\alpha', \beta'}$. However in that case we still have
\[
T(0 \alpha' \beta' e' \mid 1^r x) > \log (\size{A'_n}) - 4 = \size{1^r x} - O(1)
\]
because the set $A'_n$ constructed in the game $\Gamma_{n,\alpha', \beta'}$ from $n$ and advice $\alpha'$  has size $\Omega(2^{an})$ (otherwise $\Gamma_{n,\alpha', \beta'}$ would not have been played).  It follows that for the $x$ in winning cell,
\begin{equation}
\label{e:totalineq}
T\big[\min_\psi(1^rx) \mid 1^r x\big] > \size{1^r x} - O(1).
\end{equation}

In summary, for every $n$ there exists $x \in A_n$, namely the $x$ from the winning cell of the game $\Gamma_{n,\alpha, \beta}$ where $\alpha$ and $\beta$ are the correct advice for $A_n$ and $B_n$, for which inequality~\eqref{e:totalineq} holds. The theorem is proven.
\end{proof}

\section{$\MIN_\phi$ under computably bounded numberings}
\label{s:compbound}
We show that some recursion-theoretic results from the literature concerning minimal indices for acceptable numberings also hold for computably bounded numberings.  The following lemma refines a theorem of Blum~\cite{Blu67}.

\begin{lemma} \label{lem:immunity}
If $\phi$ is a computably bounded numbering, then $\MIN_\phi$ is immune.
\end{lemma}
\begin{proof}
Otherwise, by Lemma~\ref{lem:okra}, we could compute for every positive integer $k$ a string $x$ with $C(x) > k$, which is impossible.
\end{proof}
The next argument follows \cite[Theorem~2.11]{Sch98} which in turn credits~\cite{Mey72}.

\begin{lemma} \label{lem: MIN computes K}
For any computably bounded numbering $\phi$, $\MIN_\phi \geq_\T \0'$.
\end{lemma}
\begin{proof}
Let $\psi$ be the default acceptable numbering which $K$ is defined with respect to.  By the computably bounded property, there exists a computable function $f$ such that for any index $e$, there exists $j \leq f(e)$ defined by
\[
\phi_j (x) =
\begin{cases}
1 &\text{if $\psi_{e,x}(e) \converge$,} \\
\diverge &\text{otherwise.}
\end{cases}
\]
Let $a$ be the $\phi$-index for the everywhere divergent function.  Using a $\MIN_\phi$ oracle, compute the value 
\begin{multline*}
m(e) = \max \{s \colon \text{$s$ is the first value at which some index in} \\
\text{$\{0, 1, \dotsc, f(e)\} \cap \MIN_\phi \setminus \{a\}$ converges.}\}
\end{multline*}
Now $e \in K$ iff $\psi_{e,m(e)}(e) \converge$.  Indeed, unless $e \notin K$, $m(e)$ is an upper bound for the time required for $\psi_e(e)$ to converge.
\end{proof}

We now exploit an idea from \cite[[Theorem~11]{JST11} while bootstrapping off of Lemma~\ref{lem: MIN computes K}.
\begin{lemma} \label{lem: nolala}
For any computably bounded numbering $\phi$, $\MIN_\phi \equiv_\T \0''$.
\end{lemma}

\begin{proof}
$\MIN_\phi \leq \0''$ follows from Post's~Theorem \cite{Soa87}.  For the reverse direction, let $\psi$ be an acceptable numbering, and let $f$ be a computable function witnessing that $\phi$ is computably bounded.  We show that the set $\{\pair{d,e} \colon \psi_d = \psi_e\}$ is computable in $\MIN_\phi$.  Then $\MIN_\phi \geq \0''$  in immediate, as deciding equality in an acceptable numbering is $\Pi^0_2$-complete. \cite{Soa87}.

By Lemma~\ref{lem: MIN computes K}, our $\MIN_\phi$-computable algorithm is permitted to query $K$.  So given a pair of $\psi$-indices $\pair{d,e}$, use the $K \join \MIN_\phi$-oracle to find the unique $\phi$-minimal indices $i \leq f(d)$ and $j \leq f(e)$ such that $\phi_i = \psi_d$ and $\phi_j = \psi_e$.  Now $\psi_d = \psi_e$ iff $i=j$.
\end{proof}

Our final argument follows the idea of \cite[Theorem~2.22]{Sch98}, where Schaefer proves the same result but restricted to the case where $\phi$ is acceptable.

\begin{theorem} \label{thm: schaefer12}
For any computably bounded numbering $\phi$, $\MIN_\phi$ is not $(1,2)$-recursive.
\end{theorem}

\begin{proof}
Let $\phi$ be a computably bounded numbering, and suppose that some computable function $f: \N^2 \to \{\H,\L\}^2$ witnesses that $\MIN_\phi$ is $(1,2)$-recursive.  Here the label $\H$ asserts that a given index is minimal and $\L$ asserts that it isn't.  Let 
\[
D = \{e \colon \phi_e(x) \diverge \text{ for all $x > 0$}\}.
\]

We argue that there exists an index $x$ in the complement of $\MIN_\phi$ such that for any index $d \in D$, $f(x,d)$ either assigns the label $\H$ to $x$ or labels $d$ correctly.  Suppose this were not the case.  Then for all $x$,
\[
x \in \MIN_\phi \iff (\forall d \in D)\: [\text{$f(x, d)$ either labels x with $\H$ or labels $d$ correctly}].
\]
The forward direction follows from the definition of $(1,2)$-recursive, and the reverse direction follows from the assumption.  But since $D$ is a $\Delta^0_2$ set, $\MIN_\phi$ is now both $\Sigma^0_2$ and $\Pi^0_2$, contrary to Lemma~\ref{lem: nolala}.

Let $x$ be the distinguished nonminimal element described in the previous paragraph, and let $g(y)$ be the computable function which returns the label for $y$ in the pair $f(x,y)$.  Then $g$ gives the correct label for $y$ whenever $y \in D$.  Indeed for $d \in D$, if $f(x,d)$ assigns the label $\H$ to $x$ then by definition of $(1,2)$-recursive the label for $d$ must be correct, and otherwise the label for $d$ is correct via the special property of index~$x$.  Hence the computable set $A = \{e \colon g(e) = \H\}$ contains $\MIN_\phi \cap D$ and is disjoint from $\complement{\MIN}_\phi \cap D$.

Fix an acceptable numbering $\psi$, and let $t(e)$ be the $\phi$-index for the minimal function defined by 
\[
\phi_{t(e)}(z) =
\begin{cases}
s &\text{if $z=0$ and $s$ is the first stage at which $\psi_{e,s}$ converges on some input.}\\
\diverge &\text{otherwise,}
\end{cases}
\]
and let $h$ be a computable bound for $t$.  For every $a \in A$, with the exception of the minimal index for the everywhere divergent function, $\phi_a$ converges on at least one input because $A \subseteq \complement{D} \cup \MIN_\phi$.  Define the computable function $m$ by
\begin{multline*}
m(e) =  \max \{\phi_j(0) \colon \text{$j \in A$, $j \leq h(e)$, and} \\
\text{0 is the first value where $\phi_j$ appears to converge}\}.
\end{multline*}
Then $\psi_e$ converges on some input iff $\psi_{e,m(e)}$ does, contradicting that the fact that the index set $\{e \colon (\exists z)\: [\psi_e(z) \converge]\}$ is $\Sigma_1$-hard \cite{Soa87}.
\end{proof}

\paragraph{Acknowledgements.}  The authors thank Sanjay Jain for useful comments on the presentation of this work and are grateful to Sasha Shen and Nikolay Vereshchagin for their help with the ``warm-up'' in Section~\ref{sec:wgo}.

\bibliographystyle{alpha}
\bibliography{approximating_min}
\end{document}